\documentclass[a4paper,12pt]{amsart}

\usepackage{latexsym}
\usepackage{amssymb}
\usepackage{graphicx}
\usepackage{wrapfig}
\usepackage{amsthm}
\usepackage{float}
\usepackage{epsfig}
\usepackage{multirow}
\usepackage[toc,page]{appendix}
\usepackage[footnotesize]{caption}
\RequirePackage{color}
\usepackage{cite}
\usepackage{setspace}

\usepackage{tikz}
\usepackage{tkz-berge}
\usetikzlibrary{arrows,snakes,backgrounds,trees}
\usetikzlibrary{matrix,decorations.pathreplacing,positioning}
\usepackage{stmaryrd}
\usepackage{tkz-graph}
\usetikzlibrary{decorations.markings}
\tikzstyle directed=[postaction={decorate,decoration={markings,
    mark=at position .65 with {\arrow{stealth}}}}]
\tikzstyle reverse directed=[postaction={decorate,decoration={markings,
    mark=at position .65 with {\arrowreversed{stealth};}}}]

\numberwithin{figure}{section}
\usepackage{subfigure}

\usepackage{amscd,enumerate}
\usepackage{amsfonts}
\usepackage{enumitem} 

\input xy
\xyoption{all}

\usepackage{multicol}
\usepackage{hyperref}
\hypersetup{ colorlinks,
linkcolor=blue,
filecolor=green,
urlcolor=blue,
citecolor=blue }

\theoremstyle{plain}
\newtheorem{lemma}{Lemma}[section]
\newtheorem{corollary}[lemma]{Corollary}
\newtheorem{theorem}[lemma]{Theorem}
\newtheorem{proposition}[lemma]{Proposition}
\newtheorem{remark}[lemma]{Remark}
\newtheorem{definition}[lemma]{Definition}

\newtheorem{example}[lemma]{Example}

\newcommand{\tq}{:\hspace{4pt} }

\def\A{\mathcal{A}}

\def\span {\mathop {\rm span}\nolimits}
\def\rad {\mathop {\rm rad}\nolimits}
\def\asi {\mathop {\rm asi}\nolimits}
\def\ann {\mathop {\rm ann}\nolimits}

\def\N{{\mathbb N}}

\def\supp {\mathop {\rm supp}\nolimits}

\newcommand{\m}{\mathfrak{m}}
\newcommand{\LL}{\lambda}
\newcommand{\LLm}{\lambda_{\m}}
\begin{document}
\onehalfspacing
\subjclass[2010]{17D92, 17A60, 17A36, 05C25 }
\keywords{Genetic Algebra, Evolution Algebra, Graph, Absorption Radical} 

\title[charac. of the absorption radical of an e. a. using their associated graph]{characterization of the absorption radical of an evolution algebra using their associated graph}

\author[Paula Cadavid]{Paula Cadavid}
\address{Paula Cadavid}
\email{pacadavid@gmail.com}

\author[Tiago Reis]{Tiago Reis}
\address{Tiago Reis: Universidade Tegnol\'ogica  Federal do Paran\'a, Av. Alberto Carazzai, 1640 - Campus Corn\'elio Proc\'opio - PR, Brazil.}
\email{treis@utfpr.edu.br }

\author[Mary Luz Rodi\~no Montoya]{Mary Luz Rodi\~no Montoya}
\address{Mary Luz Rodi\~no Montoya: Instituto de Matem\'aticas - Universidad de Antioquia, Calle 67 N$^{\circ}$ 53-108, Medell\'in, Colombia}
\email{mary.rodino@udea.edu.co}

\begin{abstract}

In this paper we present a method for finding the absorbing radical of a finite-dimensional evolution algebra.  Such a method consists of finding the acyclic vertices of an oriented graph associated with the algebra. The set of generators associated with such vertices turn out to be the generators of the absorption radical.  As an application we use the absorption radical to study the decomposability of some degenerate evolution algebras.
\end{abstract}

%We characterize the absorption radical of an evolution algebra using a directed graph associated to it. Our study provides a constructive method for finding such a radical. Finally, we study the decomposability of a degenerate algebra via their absorption radical.
\maketitle
\section{Introduction}

The theory of evolution algebras is a current field of research which started in 2006 (see \cite{tian1}) when Tian and Vojtechovsky established the theoretical foundations of these structures.  Then in \cite{tian} Tian studies the connection between this type of algebra and other structures such as graphs, Markov chains, groups, dynamical systems, among others. For a review of some advances in this type of algebra we refer the reader to \cite{ceballos/falcon/nunez/tenorio/2022, reis2020, boudi20, YMV, YPT,PMP3, cabrera21, YPMP, Elduque/Labra/2015, nilpotentelduque, Elduque/Labra/2019, TeseTiago} and the references therein.

Let $\mathbb{K}$ be a field and let $\A$ be a $\mathbb{K}$-algebra. We say that $\A$ is an evolution algebra if it admits a basis $B=\{e_i\}_{i \in \Lambda} $ such that $e_i \cdot e_i  =  \sum_{k \in \Lambda} w_{ik} e_k$ for $i\in \Lambda$ and $e_i \cdot e_j = 0$, for $i,j\in\Lambda$ such that $i\neq j$. In this case the  basis $B$  is called {\it natural}, the scalars $w_{ik} \in \mathbb{K}$ the {\it structure constants} of $\mathcal{A}$ (relative to $B$), and $M_{B}=(w_{ik})$ the {\it structure matrix} of $\A$ (relative to $B$).  From now on we will consider finite-dimensional evolution algebras over a field of arbitrary characteristic. 

Evolution algebras are, in general, non-associative but are not defined by identities, like Lie algebras or Jordan algebras, therefore their study uses different tools.  The most common strategy is to fix a natural basis and, whenever possible, to obtain conclusions that are independent of the choice of such a basis.

The absorption radical was introduced in \cite{YMV} where the authors introduce and discuss also the concepts of evolution ideal and evolution subalgebra, among others. In particular, they study evolution ideals with the absorption property, and call the intersection of all these ideals the absorption radical, and denote it by $rad(\A)$.  There they show that if $\A$ is an evolution algebra, then $\rad(\A)$ is the smallest ideal of $\A$ such that $\A/\rad(\A)$ is a non-degenerate evolution algebra.  Subsequently in \cite{cabrera21} the absorption radical was characterized using the  upper annihilating  series  and the annihilator stabilizing index, $\asi(\A)$. In fact, they showed that $\rad(\A)=\ann^{(r)}(\A)$, where $r=\asi(\A)$.

The main contribution of this work is that we give a simple method to calculate the absorption radical of an evolution algebra. Then as an application we explore the relationship between such a radical and the decomposability of the algebra when it is non-degenerate. For this purpose we use as our main tool a directed graph associated to an evolution algebra. Such a graph was introduced in \cite{Elduque/Labra/2015} where it was used to study decomposability, nilpotency and automorphisms. Such a feature is also used in \cite{PMP3,YPMP, YPT, reis2020} where the authors use characteristics of the associated graph to describe the derivations of the algebra.

This paper is organized as follows. In Sec.\ref{preliminaries} we review definitions and notations about directed graphs and evolution algebras that we will use to present our results. In Sec.\ref{caracradical}, we use the associated graph to an evolution algebra in such a way that the absorption radical can be totally determined by the properties of this graph. Proposition \ref{anni=spanlambdai} 
shows how to calculate  recursively the terms of the upper annihilating series and, as a consequence of this fact, Corollary \ref{cor:rad=annm} states that to calculate the absorption radical.  Again, as a consequence of Proposition \ref{anni=spanlambdai}, we can easily compute the type of a nilpotent evolution algebra, which was defined in \cite{nilpotentelduque}.  The main result is Theorem \ref{teo:rad}, which characterize the absorption radical in terms of the acyclic vertices of the associated graph. As an application of this fact Corollary \ref{cor:rad_nilpotente} we have that the radical is a nilpotent ideal. In Sec. \ref{decomposability}, our goal is to establish a relationship between the decomposability of an evolution algebra  and the decomposability of its absorption radical.
Theorem \ref{teo:quocienteconexo} show that, under certain conditions, if $\rad(\A)$ is decomposable then $\A$ is decomposable.

\section{Preliminaries}\label{preliminaries}

% {\color{magenta}Podemos considerar aqui caminhos com vertices repetidos?}

We start with basic definitions and notation for directed graphs. A {\it directed graph} is a pair $G=(V,E)$ where $V$ is a set and $E\subseteq V\times V$. The elements of $V$ are called the {\it vertices} of $G$ and the elements of $E$ the {\it arrows} or {\it directed edges} of $G$.  If $V$ and $E$ are both finite, we say that $G$ is {\it finite}. In this case, the {\it adjacency matrix} of $G$ is the matrix $A_{G}=(a_{ij})$ where $a_{ij}=1$ if $(i,j) \in E$ and $a_{ij}=0$ if $(i,j) \not\in E$. In this paper, we will consider only finite graphs.

Let $G=(V,E)$ a graph. We say that $H=(V',E')$ is a subgraph of  $G$ if $V'\subseteq V$ and $E'\subseteq E$. Furthermore, if $H=(V',E')$ is a subgraph of  $G$ such that if  $v,w\in V'$ and $(v,w) \in E$ then $(v,w) \in E'$, we say that $H$ is a {\it full subgraph} of $G$. Let $u,v \in V$. A {\it path} $\mu$ from $u$ to $v$ is a finite sequence of vertices $\mu= v_0v_1\dots v_{n-1}v_n$ such that $v_0=u$, $v_n=v$ and $(v_i,v_{i+1})\in E$ for all $i\in \{0,1,\dots,n-1 \}$. In this case we say that $n$ is the {\it length} of the path $\mu$ and denote by $\mu^0$ the set of its vertices, i.e.,
$\mu^0=\{v_0,v_1,\dots,v_n\}$. 
If $\mu = v_0v_1 \dots v_n$ is a path and $v_0=v_n$ then $\mu$ is called a {\it cycle based at} $v_1$. A cycle of length $1$ will be said to be a {\it loop}.
If $\theta$ is a cycle and there exist a path from $w\in V$ to $ z\in \theta^0$,  we say that  there exist  {\it a path from $w$ to $\theta$} and that $w$ is a {\it cyclic vertex}, otherwise we say that $w$ is an {\it acyclic vertex}. Let $v\in V$. The {\it first-generation descendants} of $v$ are the elements of the set $D^1(v)$ given by
$D^1(v)=\{u \in V \tq (v,u)\in E\}.$ 
If $V' \subseteq V$ we will denote the first-generation descendants of $V'$ by $\displaystyle D^1(V')=\{w\in V \tq w\in D^1(u) \text { for some } u\in V' \}.$ By recurrence, we define the set of {\it mth-generation descendants} of $v$ by $D^{m}(v)=D^1(D^{m-1}(v)),$ for  $m \in \N^* $. Finally, the set of {\it descendants} of $v$ is defined by  $$\displaystyle D(v)=\bigcup_{m \in \N^*} D^m(v).$$ 
A graph $G=(V,E)$ is called {\it connected}  if for all partition $V=V_1\cup V_2$ there are $x \in V_1$ and $y\in V_2$ such that $(x,y)\in E$ or $(y,x)\in E$.

%%%%%% preliminares algebras de evolucion

There appear several ways of associating an oriented graph to an evolution algebra, see for example \cite{YMV, Elduque/Labra/2015}, we will consider the one used in \cite{YMV}. Given a natural basis $B=\{e_i\}_{i \in \Lambda}$ of an evolution algebra $\A$ and its structure matrix $M_B=(\omega_{ij})$, consider the matrix
$P=(a_{ij})$ such that $a_{ij}=0$ if $\omega_{ij}=0$ and $a_{ij}=1$ if $\omega_{ij}\neq 0$. The {\it graph associated to the evolution algebra} $\A$ (relative to the basis $B$), denoted by $\Gamma(\A,B)$ is the directed graph whose adjacency matrix is given by $P=(a_{ij})$. At this point it is important to note that the graph associated with an evolution algebra $\A$ depends on the natural basis.  Examples illustrating this dependence can be found in \cite[Example 2.5]{Elduque/Labra/2015} and \cite[Example 2]{YPMP}.

Let $\A$ be an  evolution algebra  and $B=\{ e_1, \ldots, e_n\}$ a natural basis. If $u=\sum_{i=1}^n x_ie_i \in \A$ the the {\it support of $u$ relative to $B$} is defined as $\supp_B(u)=\{i \tq x_i\not=0\}$. We say that $\A$ is {\it decomposable} if there are non-zero ideals $I$ and $J$ such that $\A= I\oplus J$. Otherwise, $\A$ is called {\it indecomposable}. The {\it  annihilator} of $\A$ is $\ann(\A)=\{ x\in A \tq x\A=\{0\}\}$. In \cite{Elduque/Labra/2015}, the authors proof that $\ann(\A)= \span\{e_i \tq e_i^2=0\}$, which means that $\ann(\A)$ is generated by the elements of $B$ corresponding to the sources of $\Gamma(\A,B)$. If $\ann(\A)=\{0\}$ we say that $\A$ is {\it non-degenerate}. It is well known that the indecomposability of $\A$ and the connectedness of $\Gamma(\A,B)$ are related. In fact,  by \cite[Proposition 2.10]{Elduque/Labra/2015},  $\A$ is indecomposable if and only if the graph $\Gamma(\A,B)$ is connected for any natural basis $B$ of $\A$. Moreover, if $\A$ is non-degenerated and there is a natural basis $B$ such that $\Gamma(\A,B)$ is connected then $\A$ is indecomposable \cite[Proposition 2.8]{Elduque/Labra/2015}.

Let $\A'$ be a subalgebra of an evolution algebra $\A$. 
We say that $\A'$ is an \textit{evolution subalgebra} if $\A'$ is an evolution algebra. Analogously,  we say that $I$ is an \textit{evolution ideal}  if is and ideal  of $\A$ that has a natural basis.  Here we are using the definitions given in \cite{YMV}, not the definitions given in \cite{tian}, and refer the reader to \cite{YMV} for details on how these  definitions differ from those given in \cite{tian}. Again following  \cite{YMV}, we say that a  evolution  subalgebra $\A'$ has the \textit{extension property} if there exists a natural basis $B$  of $\A'$ which can be extended to a natural basis of $\A$.

\begin{definition} \rm
Let $I$ be an ideal of an evolution algebra $\A$. We will say that $I$ has the {\it absorption property} if $x\A\subseteq I$ implies $x\in I$.
\end{definition}
Equivalently, an ideal $I$ has the property of absorption if $\ann(A/I)=0$. In addition, if it has the absorption property, then it also has the following characteristic.

\begin{proposition}\cite[Lemma 2.25]{YMV} \label{prop:extension_property}Let $\A$ be an evolution algebra with a natural basis $B= \{e_i\}_{i \in \Lambda}$ and $I$ be a non-zero ideal. If $I$ has the absorption property then there  is $\Lambda' \subseteq \Lambda$ such that $I=\span\{e_i\}_{i \in \Lambda'}$. In particular, $I$ has the extension property.
\end{proposition}

Now we present our main object of study.
\begin{definition} \rm
Let $\A$ be  an evolution algebra and let $\Psi (\A)$ be  the class of all ideals of $\A$ having the absorption property. The  \textit{absorption radical} of $\A$, denoted by $\rad(\A)$, is defined by
$$\displaystyle \rad (\A)=\bigcap_{I \in \Psi (\A)} I .$$
\end{definition}
Clearly the  absorption radical is an ideal  with the absorption property.
%\begin{remark}
%Let $\A$ be an evolution algebra with a natural basis $B= \{e_i\}_{i \in \Lambda}$. As $\rad(\A)$ has the absorption property then, by \cite[Lemma 2.25]{YMV}, there exist $\Lambda' \subseteq \Lambda$ such that $\rad(\A)=\span\{e_i\}_{i \in \Lambda'}$. 
%\end{remark}
%%%%%%% Olhando ainda como encaixar articuladamente os seguintes conceitos

To establish our results we will use the upper annihilating series and the annihilator stabilizing index (which was introduced in \cite{cabrera21}), which are defined for commutative algebras as follows. Let $\A$ be a commutative algebra. We define $\ann^{(0)}(\A)=\{0\}$ and $\ann^{(i)}(\A)$ in the following way 
$$\ann^{(i)}(\A)/\ann^{(i-1)}(\A)=\ann(\A/(\ann^{(i-1)}(\A))). $$
The chain of ideals
$$\{0 \}=\ann^{(0)}(\A)\subseteq \ann^{(1)}(\A)\subseteq \ldots \subseteq  \ann^{(i)}(\A) \subseteq \dots $$
is called the {\it upper annihilating series}. If there exists $k$ such that $k=\min \{q: \ann^{(q)}(\A)= \ann^{(q+1)}(\A)\}$, then we call it the \textit{annihilator stabilizing index} of $\A$, denoted by $\asi(\A)$. We emphasize that if $\A$ is a finite dimensional algebra then $\asi(\A)$ exists. Finally, if  $\A$ is  a nilpotent  evolution algebra and $r=\asi(\A)$ then the  {\it type} of $\A$ is the sequence $[n_1, \dots, n_r]$ such that $n_1+\dots+ n_i=\dim (\ann^{(i)}(\A))$ for all $i\in \{1,\dots, r\}$. In other words, 
$$n_i=\dim(\ann(\A/\ann^{(i-1)}(\A)))=\dim(\ann^{(i)}(\A))-\dim(\ann^{(i-1)}(\A)).$$

\section{Characterization of the absorption radical} \label{caracradical}

The main purpose of this section is to present Theorem \ref{teo:rad} which provides a description of the absorption radical of an evolution algebra using its associated graph.  We begin our discussion by pointing out that, as \cite[Example 2.26]{YMV} shows, the reciprocal of Proposition \ref{prop:extension_property} is not valid. However, our first result provides a sufficient and necessary condition for an ideal with the extension property to have also the absorption property.

%The main goal of this section is to present the Theorem \ref{teo:rad}. We started by noting that \cite[Lemma 2.25]{YMV} shows that if and ideal $I$  of an evolution algebra has the absorption property, then $I$ has the extension property as well. However, the reciprocal is not true, as shown in \cite[Example 2.26]{YMV}. The next result presents conditions for the reciprocal to be true.

\begin{proposition} \label{prop:PAsessD(j)nsubLambdaI}
Let $\A$ be an evolution algebra, $I$ be  an ideal having the 
extension property and  $B=\{e_i\}_{i \in \Lambda}$ a natural basis of $\A$ such that $I=\span \{ e_i\}_{i \in\Lambda'}$ with $\Lambda' \subseteq \Lambda$. Then $I$ has the absorption property if, and only if, 
\begin{equation} \label{eq:D(j)subLambdaI}
 D^1 (j) \not\subseteq \Lambda', \text{ for all } j \in \Lambda \setminus\Lambda'.
\end{equation}
\end{proposition}
\begin{proof}
Let  $M_B=(w_{ij})$ be the structure matrix of  $\A$ relative to $B$. Suppose that  $I$  has the absorption property and, contrary to our claim, that there is  $j\in \Lambda\setminus \Lambda'$ such that $D^1(j) \subseteq \Lambda'$. If $x=\sum_{k \in \Lambda} x_ke_k \in \A$ then   $$xe_j=\left( \sum_{k \in \Lambda} x_ke_k\right)e_j=x_je_j^2=x_j\left(\sum_{k \in D^1(j)}w_{jk}e_k\right).$$
As $D^1(j)\subseteq \Lambda'$ then $e_jx \in I$ and consequently $e_j\A \subseteq I$. Thus $e_j \in I$, which is a contradiction, since that  $j\in \Lambda\setminus \Lambda'$. Conversely, suppose that $I=\span \{ e_i\}_{i \in\Lambda'}$ is an ideal having the extension property that satisfies Eq. \eqref{eq:D(j)subLambdaI} and
let $x=\sum_{k \in \Lambda} x_ke_k \in \A$ such that  $x\A\subseteq I$. If  $x \not\in I$  then there exist $j \in \Lambda\setminus \Lambda'$ such that $x_j\neq 0$.
Therefore
$$\displaystyle xe_j=x_je_j^2=x_j\left(\sum_{k \in \Lambda}w_{jk}e_k\right)=x_j\left(\sum_{k \in D^1(j)}w_{jk}e_k\right)\in I.$$ 
Then  $\supp_B ( xe_j)=D^1(j) \subseteq \Lambda'$, a contradiction since that $I$ satisfies Eq.\eqref{eq:D(j)subLambdaI}. Therefore $x \in I$, i.e., $I$ has the absorption property.
\end{proof}

\begin{corollary} \rm \label{cor:propriedades}
Let $\A$ be an evolution algebra with a natural basis $B= \{e_i\}_{i \in \Lambda}$, $\rad(\A)=\span\{e_i\}_{i \in \Lambda'}$ with $\Lambda' \subseteq \Lambda$. If  $j \in \Lambda$ and 
$D^n(j)\subseteq \Lambda' $ for some $n \in \mathbb{N}^{*}$, then $j \in  \Lambda'$.
\end{corollary}
\begin{proof}
%Let $j\in  \Lambda$. If $D^1(j)\subseteq \Lambda'$ and $j \not \in \Lambda' $ by Proposition \ref{prop:PAsessD(j)nsubLambdaI} we have that  $\rad(\A)$ has not the absorption property, which is a contradiction. If $D^n(j)\subseteq \Lambda'$ for $n>1$  since by definition  $D^n(k)=D^1(D^{n-1}(k))$ then $D^1(k)\subseteq \Lambda'$ for every $k \in D^{n-1}(j)$. Consequently $D^{n-1}(j) \subseteq \Lambda'$. Continuing with this process our statement follows.

Let $j\in  \Lambda$. For $n=1$ the statement follows directly from Proposition \ref{prop:PAsessD(j)nsubLambdaI} because $\rad(\A)$ is an ideal with the absorption property. If $D^n(j)\subseteq \Lambda'$ for $n>1$, since  $D^n(j)=D^1(D^{n-1}(j))$ then $D^1(k)\subseteq \Lambda'$ for $k \in D^{n-1}(j)$. Therefore $D^{n-1}(j) \subseteq \Lambda'$. Continuing with this process our statement follows.
\end{proof}

\begin{remark}\rm \label{obs:deflambda}
Let  $\A$ be an evolution algebra with a natural basis   $B=\{e_i\}_{i \in \Lambda}$. We define the following subsets of $\lambda_1(B)=\{i \in \Lambda \tq e_i^2=0 \}$  and $\lambda_k(B)=\{i \in \Lambda \tq D^1(i)\subseteq \lambda_{k-1}(B)\}$, for all $k > 1.$
Let us see that $\lambda_{k}(B)\subseteq \lambda_{k+1}(B)$ for $k \in \mathbb{N}^{\star}$.  If $i \in \lambda_1(B)$, then $D^1(i)=\emptyset$ and consequently  $D^1(i) \subseteq \lambda_1(B)$.
Therefore $\lambda_1(B)\subseteq \lambda_2(B)$. 
Suppose that $\lambda_{t-1}(B)\subseteq \lambda_t(B)$  for some  $t>1$. If $j \in \lambda_t(B)$ then  $D^1(j)\subseteq \lambda_{t-1}(B) \subseteq \lambda_t(B),$
which means that $j\in \lambda_{t+1}(B)$ and therefore  $\lambda_t(B)\subseteq \lambda_{t+1}(B)$. In this way we have constructed the following  increasing sequence of subsets of $\Lambda$ 
\begin{equation} \label{eq:lambdaj}
    \lambda_1 (B)\subseteq \lambda_2(B)  \subseteq \dots \subseteq \lambda_k(B) \subseteq \dots \subseteq \Lambda.
\end{equation}

\end{remark}

\begin{proposition}
\label{anni=spanlambdai}
Let  $\A$ be an evolution algebra with a natural basis   $B=\{e_i\}_{i \in \Lambda}$. Then $$\ann^{(i)}(\A)=\span \{e_j \tq j \in \lambda_i (B) \} \text{ for all } i \in \N^*.$$
\end{proposition}

\begin{proof}
First,  we observe that  
$\ann^{(i)} (\A)= \{x \in \A: x\A \cup \A x \subseteq \ann^{(i-1)}(\A)\}\text{ for } i\geq 1$, by \cite[Lemma 3.5 (i)]{cabrera21}. We will prove our statement by induction. For $i=1$ we have that $\ann^{(1)}(\A)=\ann(\A)=\span\{e_j \tq e_j^2=0\}=\span\{e_j \tq j \in \lambda_1(B) \}$.
Suppose that $\ann^{(i)}(\A)=\span \{e_j \tq j \in \lambda_i (B)\}$ for some $i \in \N^*$. We will prove that  $\ann^{(i+1)}(\A)\subseteq \span \{e_j \tq j \in \lambda_{i+1}(B)\}$. 
Let  $x=\sum_{t \in \Lambda'}x_te_t\in \ann^{(i+1)}(\A)$, where $\Lambda' =\supp_B (x)$. 
If $k \in \Lambda'$ then
$$xe_k=\left(\sum_{t \in \Lambda'} x_t e_t\right)e_k=x_ke_k^2=x_k\left(\sum_{\ell \in D^1(k)}w_{k\ell}e_{\ell} \right) \in \ann^{(i)}(\A)=\span \{e_j \tq  j \in \lambda_{i}(B) \},$$
namely $D^1(k) \subseteq \lambda_k(B)$ and therefore $k \in \lambda_{i+1}(B)$. 
We conclude that  $x \in \span \{e_j \tq j \in \lambda_{i+1} (B)\}$. To prove  the other inclusion,  that is  that $\span \{e_j \tq j \in \lambda_{i+1}(B)\} \subseteq \ann^{(i+1)}(\A)$,  let  $  x =\sum_{t\in \lambda_{k+1}(B)}x_te_t\in\span\{e_j \tq j \in \lambda_{i+1}(B) \}$ and  $ y =\sum_{t\in \Lambda } y_te_t\in \A$. Then
$$xy=\left(\sum_{t\in \lambda_{i+1}(B)}x_te_t \right)\left( \sum_{t\in \Lambda } y_te_t\right)=\sum_{t\in \lambda_{i+1}(B)}x_ty_te_t^2=\sum_{t\in \lambda_{i+1}(B)}x_ty_t\left(\sum_{\ell \in D^1(t)}w_{t\ell}e_{\ell}\right).$$
As  $t \in \lambda_{k+1}(B)$ then $D^1(t)\subseteq \lambda_k(B)$. Therefore $xy \in \span \{e_t \tq t \in \lambda_k(B)\}=\ann^{(i)}(\A)$.  Thus  $x\A\subseteq \ann^{(i)}(\A)$ and, consequently,  $x\in \ann^{(i+1)}(\A)$.
\end{proof}

Our next result shows how to determine the absorption radical of an evolution algebra using only the associated graph structure. This result will be improved in Theorem \ref{teo:rad}.

\begin{corollary}\label{cor:rad=annm}
Let $\A$ be an evolution algebra with a natural basis  $B=\{e_i\}_{i \in \Lambda}$ and $\m=\asi(\A)$. Then $$\rad (\A)=\span \{e_i \in B \tq i \in \lambda_{\m}(B) \}.$$
\end{corollary}
\begin{proof}
As $\A$ is a finite dimensional algebra then there exist $\m\in \mathbb{N}$ such that  $\m=\asi(\A)$. On the other side, by \cite[Proposition 3.7 (ii)]{cabrera21}, we have that $\rad(\A)=\ann^{(\m)}(\A)$ and using Proposition \ref{anni=spanlambdai} we get our statement.
\end{proof}

\begin{example} \label{exa:radical}
Let $\A$ be an evolution algebra and $B=\{e_1, e_2,e_3,e_4,e_5\}$ a natural basis such that $e_1^2=0$, $e_2^2=e_1$, $e_3^2=e_4$, $e_4^2=e_5$ and $e_5^2=e_3$.  The graph $\Gamma(\A, B)$ presented in  Figure \ref{fig:exarad}.
\begin{figure}[!h]
\centering
\begin{tikzpicture}[scale=1]
\draw [thick] (0,0) circle (7pt);
\draw (0,0) node[font=\footnotesize] {1};
\draw [thick] (2,0) circle (7pt);
\draw (2,0) node[font=\footnotesize] {2};
\draw [thick] (4,0) circle (7pt);
\draw (4,0) node[font=\footnotesize] {3};
\draw [thick] (5.5,1) circle (7pt);
\draw (5.5,1) node[font=\footnotesize] {4};
\draw [thick] (5.5,-1) circle (7pt);
\draw (5.5,-1) node[font=\footnotesize] {5};
\draw [thick, directed]  (1.75,0) to (0.25,0)  ;
\draw [thick, directed] (3.75,0) to (2.25,0)   ;
\draw [thick, directed] (4.18,0.18) to  (5.32,0.82);
\draw [thick, directed] (5.5,0.75) to  (5.5,-0.75);
\draw [thick, directed] (5.32,-0.82) to  (4.18,-0.18);
\end{tikzpicture}
\caption{$\Gamma (\A,B)$ for the algebra $\A$ of Example \ref{exa:radical}.}
\label{fig:exarad}
\end{figure}

\noindent Using Remark \ref{obs:deflambda} we have that 
$\lambda_1(B)=\Lambda_0=\{1\} \text{ and } \lambda_2(B)=\{1,2 \}=\lambda_t(B)\text{ for all } t > 2.$
Therefore $\asi (\A)=2$ and by Corollary \ref{cor:rad=annm} we obtain that $\rad(\A)=\ann^{(2)}(\A)=\span\{e_1,e_2 \}.$
\end{example}

\begin{corollary}
Let $\A$ be a nilpotent evolution algebra with a natural basis $B=\{e_i\}_{i \in \Lambda}$,  $\m=\asi(\A)$ and $[n_1, \dots, n_{\m}]$ the type of $\A$. Then  
$$n_i=|\LL_i(B)|-|\LL_{i-1}(B)|\text{ for all }i \in \{1,\dots,\m \}.$$
\end{corollary}
\begin{proof}
By Proposition \ref{anni=spanlambdai} we have that $\dim (\ann^{(i)}(\A))=|\LL_i(B)|$ and the statement follows.
\end{proof}

\begin{corollary} \label{cor:Lmciclo}
Let  $\A$ be an evolution algebra with a natural basis $B=\{e_i\}_{i \in \Lambda}$ and  $\m=\asi(\A)$.  Then $\lambda_{\m}(B)$ satisfies the following conditions:
\begin{enumerate}[label=(\roman*)]
    \item If $\mu$ is a cycle of $\Gamma(\A,B)$, then $ \mu^0 \cap  \lambda_{\m}(B) = \emptyset$. \label{cor:Lmciclo1} 
    \item If $i \in \Lambda$ is  a cyclic vertex, then $i \not\in \lambda_{\m}(B) $. \label{cor:Lmciclo2}
    \item If $i \in \lambda_{\m}(B) \setminus \lambda_1(B)$, then there exist a path from  $i$ to some $j \in \lambda_1(B)$. \label{cor:Lmciclo3}
\end{enumerate}
\end{corollary}
\begin{proof}
\ref{cor:Lmciclo1} Let $\mu $ be a cycle of  $\Gamma (\A,B)$. Note that, since all elements of $\lambda_1(B)$ are sinks in $\Gamma (\A,B)$ then  $\mu^0\cap \lambda_1(B)=\emptyset$. 
Suppose that  $\mu^0\cap \lambda_{\m}(B) \neq \emptyset$ and let  $k=\min\{t \in \mathbb{N}^* \tq \lambda_t(B) \cap \mu^0 \neq \emptyset\}$. Then  there is  $v \in \mu^0\cap \lambda_k(B)$ and 
therefore  $D^1(v) \subseteq \lambda_{k-1}(B)$ and  consequently $\mu^0\cap \lambda_{k-1}(B) \neq \emptyset$, which is a contradiction by minimality of $k$. Thus $\mu^0\cap \lambda_{\m}(B) = \emptyset$. 
\noindent \ref{cor:Lmciclo2} Let $i$ be a cyclic vertex. Then there is $j\in \Lambda$ such that $j \in \mu^0$ where $\mu$ a cycle and a path $\theta$ from $i$ to $j$ in $\Gamma (\A,B)$. By  item \ref{cor:Lmciclo1}  we know that $j \not\in \lambda_{\m}(B)$.  Let us assume that $i\in \lambda_{\m}(B)$. Since that $\rad(\A)$ has the extension property, we have that $D(i)\subseteq \lambda_{\m}(B)$. Therefore $j \in D(i)\subseteq \lambda_{\m}(B)$, which is a contradiction.

%Let  $i,j \in \Lambda$, $\theta$ be a path from $i$ to $j$ and $\mu$ a cycle such that  $j \in \mu^0$ . 
%By  item \ref{cor:Lmciclo1} $j \not\in \lambda_{\m}(B)$. Suppose that $i\in \lambda_{\m}(B)$. Since that $\rad(\A)$ has the extension property, we have that $D(i)\subseteq \lambda_{\m}(B)$. Therefore $j \in D(i)\subseteq \lambda_{\m}(B)$, which is a contradiction. 

\noindent \ref{cor:Lmciclo3} If $i \in \lambda_{\m}(B) \setminus \lambda_1(B)$, by \ref{cor:Lmciclo2}, there is no path from $i$ to a cycle of $\Gamma (\A,B)$. On the other hand, by Remark \ref{rem:propciclico} there exist  $k\in \N^*$ such that  $D^k(i)=\emptyset$. If $t=\min \{k \in \N^*\tq D^k(i)=\emptyset\}$ and  $ j \in D^{t-1}(i)$ then there exist a path from  $i$ to $j$ and $D^1(j)=\emptyset$, i.e. $j \in \lambda_1(B)$.
\end{proof}

\begin{remark} \label{rem:propciclico}
%Let $G=(V,E)$ be a \textcolor{magenta}{directed} graph and $i \in V$. We claim that  $D^k(i)\neq \emptyset $ for  all $k \in \N^*$ if and only if there exist a path from $i$ to a cycle of $G$ (\textcolor{magenta}{Nâo é melhor dizer aqui que i é um vertice ciclico?}). Indeed, as $V$ is finite, then there are distinct  $k_1,k_2 \in \N^*$ such that $D^{k_1}(i) \cap D^{k_2}(i) \neq \emptyset$. Without loss of generality, we can assume that $k_1 < k_2$. Let $j \in D^{k_1}(i) \cap D^{k_2}(i)$. Then there exist a path $\mu$ from $i$ to $j$. Considering that $j\in D^{k_2}(i)= D^{k_2-k_1}(D^{k_1}(i))$ and  $j\in D^{k_1}(i)$, then there exist a path form $j$ to $j$, as required. %On the other hand, let $\theta=v_0v_1v_2\dots v_{n-1}v_n$ be a cycle of $G$ such that there exist a path  $\mu=u_0u_1u_2\dots u_t$, with $u_0=i$, and $u_t \in \theta^0$. Take $k\ \in \N^{*}$. If $k \leq t$, then $u_k\in D^k(i)\neq \emptyset$. If $k > t$ then $v_j \in D^k(i)$ for some $j \in \{0,1,\dots, n\}$ and so $D^k(i) \neq \emptyset$. 
%\textcolor{magenta}{Aqui estou trabalhando em deixar a observação mais compacta}

Let $G=(V,E)$ be a graph and $i \in V$. We claim that  $D^k(i)\neq \emptyset $ for  all $k \in \N^*$ if and only if $i$ is a cyclic vertex. Indeed, if  $D^k(i)\neq \emptyset $ for  all $k \in \N^*$ then, as $V$ is finite, there are $k_1,k_2 \in \N^*$,  $k_1 < k_2$  such that $D^{k_1}(i) \cap D^{k_2}(i) \neq \emptyset$. If $j \in D^{k_1}(i) \cap D^{k_2}(i)$ then  $j\in D^{k_2}(i)= D^{k_2-k_1}(D^{k_1}(i))$ and  $j\in D^{k_1}(i)$, therefore there exist a path form $j$ to $j$, as required. For the reciprocal, let $k\ \in \N^{*}$ and  $\theta=v_0v_1 \dots v_{n-1}v_n$ be a cycle of $G$ such that there is a path  $\mu=u_0u_1\dots u_t$, with $u_0=i$, and $u_t \in \theta^0$.  If $k \leq t$ then $u_k\in D^k(i)\neq \emptyset$, otherwise  if $k > t$ then $v_j \in D^k(i)$ for some $j \in \{0,1,\dots, n\}$, and $D^k(i) \neq \emptyset$ as claimed. 
\end{remark}

\begin{theorem} \label{teo:rad}
Let $\A$ be an evolution algebra with a natural basis  $B=\{e_i\}_{i \in \Lambda}$, $\m=\asi(\A)$ and $\rad(\A)=\span \{e_i\}_{i \in \LLm }$. Then $i \in \lambda_{\m}(B)$ if, and only if, $i$ is an acyclic vertex. Therefore, $$\rad(\A)= \span \{e_i \tq i \text{  is an acyclic vertex of }\Gamma(\A,B)\}.$$
\end{theorem}

\begin{proof}
If $i \in \lambda_{\m}(B)$ then $i$ is an acyclic  vertex by Corollary \ref{cor:Lmciclo} \ref{cor:Lmciclo2}. Conversely, suppose that $i \in \Lambda$ is  acyclic. Then  Remark \ref{rem:propciclico} say that there exists $k \in \N^*$ such that $D^k(i)=\emptyset$. If we define $r=\min \{k \in \N^* \tq D^k(i)=\emptyset \}$  then $D^{r-1}(i)\neq \emptyset$ and $D^1(D^{r-1}(i))= \emptyset$. Therefore  $D^{r-1}(i) \subseteq \lambda_1(B) \subseteq \lambda_{\m}(B)$ and by Corollary \ref{cor:propriedades} we obtain that $i \in \lambda_{\m}(B)$.
\end{proof}

\begin{example}
Let us consider again algebra $\A$ from Example \ref{exa:radical}.
 As $\{ 1, 2\}$ is the set of acyclic vertices of $\Gamma(\A,B)$, then by Theorem \ref{teo:rad} we obtain that $\rad(\A)=\span \{e_1,e_2 \}$.
\end{example}
\begin{corollary} Let $\A$ be an evolution algebra. Then $\rad(A)$ is nilpotent.
\end{corollary}
\begin{proof} Follows directly from Theorem \ref{teo:rad}  and \cite[Theorem 3.4]{Elduque/Labra/2015}.
\end{proof}
\begin{corollary}\label{cor:a=radnilpotente}
Let be $\A$ an evolution algebra with a natural basis $B$. Then the following conditions are equivalents: \vspace{-6pt}
\begin{enumerate}[label=(\roman*),noitemsep]
    \item $\A=\rad (\A)$. \label{cor:radnil1}
    \item $\Gamma (\A, B)$ has no cycles. \label{cor:radnil2}
    \item $\A$ is nilpotent.\label{cor:radnil3}
\end{enumerate}
\end{corollary}
\begin{proof}
The equivalence of \ref{cor:radnil1} and \ref{cor:radnil2} follows from Theorem \ref{teo:rad} and the equivalence of  \ref{cor:radnil2} and \ref{cor:radnil3}
follows from \cite[Theorem 3.4]{Elduque/Labra/2015}.
\end{proof}

\begin{corollary}\label{cor:rad_nilpotente}
Let $\A$ be an evolution algebra with a natural basis $B=\{e_i\}_{i \in \Lambda}$. Then $\rad (\A)=\ann (\A)$ if, and only if, $i$ is cyclic for every $i \in \Lambda\setminus \lambda_1(B)$.
\end{corollary}

\begin{proposition}\label{prop:EPthennil}
Let $\A$ be an evolution algebra and let  $I$ be an ideal having the extension property. Then $I$ is nilpotent if and only if $I \subseteq  \rad (\A)$.
\end{proposition}
\begin{proof} 
As $I$ has the extension property then there exists natural basis $B=\{e_i\}_{i \in \Lambda }$ of $\A$ and $B'=\{e_i\}_{i \in \Lambda' }$ of $I$ such that $\Lambda'\subseteq \Lambda$. If $I$ is nilpotent then the graph $\Gamma(I,B')$ has no cycles (see \cite[Theorem 3.4]{Elduque/Labra/2015}). Therefore using  Theorem \ref{teo:rad}  we have that $\Lambda' \subseteq \lambda_{\m}(B)$,  where $\m=\asi(\A)$.  The other implications follows from the fact that $\rad(\A)$ is nilpotent by Corollary \ref{cor:rad_nilpotente}.
\end{proof}

The hypothesis of that the ideal $I$ has the extension property cannot be eliminated in Proposition \ref{prop:EPthennil}, as we will see in the following example.

\begin{example} \label{exa:radical2}
Let $\A$ be an evolution algebra and let $B=\{e_1, e_2,e_3\}$ a natural basis such that $e_1^2=e_2^2=-e_3^2=e_2+e_3.$
Consider the ideal $I=\span \{e_2+e_3\}$ (that has not the extension property) and note that $I$ is nilpotent, since that $I^2=0$. Finally, observe that $I\not\subseteq \rad (\A)$ because $\ann (\A)=\rad (\A)=\{ 0\}$.
\end{example}

\section{Decomposability of degenerate evolution algebras} \label{decomposability}
In this section we are going to present some results that characterize decomposability of degenerate evolution algebras, under certain conditions, using the absorption radical. 
If $\A$ is a non-degenerate evolution algebra then, by \cite[Propostion 2.8]{Elduque/Labra/2015}, their decomposability can be determined studying the connectedness of the associated graph. Furthermore, by \cite[Propostion 2.10]{Elduque/Labra/2015},  if $\A$ is degenerate and has a natural basis $B$ such that $\Gamma(\A,B)$ is not connect, then $\A$ is decomposable. So we will focus on degenerate evolution algebras with a natural basis such that the associated graph is connected.

%\begin{definition}
%An algebra $\A$ is said to be { \it decomposable} %if there are non-zero ideals $I$ and $J$ such that %$\A=I\oplus J$. Otherwise, $\A$ is called {\it indecomposable}.
%\end{definition}

If $I$ is an evolution ideal with extension property of an evolution algebra $\A$, the next proposition establish a relation between the structure matrices of $\A$,  $I$ and $\A/I$. Furthermore, a consequence of such relation is presented in Proposition \ref{prop:relacaoentregrafos}, when $I=\rad(\A)$.

\begin{proposition}\label{prop:matrizblocosrad}
Let $\A$ be an evolution algebra, $I$ an ideal having the extension property and $B=\{e_i\}_{i \in \Lambda}$ a natural basis of $\A$ such that $I=\span \{ e_i \tq i \in \Lambda'\}$ and $\Lambda' \subseteq \Lambda$. 
If $B$ is reordered in such a way that $i < j$ for all $i\in \Lambda'$ and $j \in\Lambda\setminus \Lambda'$, then there exist natural basis  $B'$ and $\overline{B}$ of $I$ and $\A/I$, respectively, such that $M_B$ has the shape:
\begin{equation}
\label{eq:matrizblocosrad}
M_B=\left(
\begin{array}{cc}
    M_{B'} & 0 \\
    X & M_{\overline{B}}
\end{array}
\right),\end{equation}
where $M_{B'}$ is the structure matrix of  $I$  related to $B'$ and $M_{\overline{B}}$ is the structure matrix of  $\A/I$ related to $\overline{B}$.
\end{proposition}
\begin{proof} Let $B$ be a basis satisfying  the conditions of our hypothesis  and $M_B=(w_{ij})_{i,j \in \Lambda}$ the corresponding structure matrix of $A$.
As $I$ is an ideal of $\A$ then $ e_i^2=\sum_{k\in \Lambda'}w_{ik}e_k$ for all $i \in \Lambda'$. Therefore the first $|\Lambda'|$ rows of $M_B$ has the desired form. On the other hand,  as $\overline{B}=\{\overline{e_i}\tq i \in\Lambda\setminus \Lambda' \}$ is a natural basis of $\A/I$ we have that 
\begin{equation} \label{eq:ei^2=wij=zij}
\overline{e_i}^2=\overline{\sum_{k \in \Lambda} w_{ik}e_k}=\sum_{k \in \Lambda} w_{ik}\overline{e_k}=\sum_{k \in\Lambda\setminus \Lambda'} w_{ik}\overline{e_k}, \, \text{ for all }i \in\Lambda\setminus \Lambda'.\end{equation}
If $M_{\overline{B}}=(z_{ij})_{i,j\in\Lambda\setminus\Lambda'}$ is the structure matrix of  $\A/I$ related to $\overline{B}$, then 
$$\overline{e_i}^2=\sum_{k \in\Lambda\setminus \Lambda'} z_{ik}\overline{e_k}\text{ for all }i \in\Lambda\setminus \Lambda'.$$
By Eq. \eqref{eq:ei^2=wij=zij} we have $z_{ij}=w_{ij}$ for all $i,j \in\Lambda\setminus \Lambda'$.
\end{proof}

\begin{corollary}\label{prop:relacaoentregrafos}
Let $\A$ be an evolution algebra with a natural basis $B$.  Then there exists  natural basis  $B'$ and $\overline{B}$ of $\rad (\A)$ and $\A/\rad(\A)$, respectively,  such that 
$\Gamma (\rad(\A),B')$ and $\Gamma (\A/\rad(\A),\overline{B})$
are full subgraphs  of $\Gamma (\A,B)$.
\end{corollary}
\begin{proof} 
Consider the natural basis $B'=\{e_i \tq i \in \LL_{\m}(B) \}$ of $\rad(\A)$. As $\rad (\A)$ has the extension property we can use Proposition \ref{prop:matrizblocosrad} considering the basis  $\overline{B}=\{\overline{e_i}\tq i \in\Lambda\setminus \Lambda' \}$ of $\A/\rad(\A)$. Then $M_B$ is seen as in Eq. \eqref{eq:matrizblocosrad}  and our statement follows.
\end{proof}
\begin{lemma} \label{lema:radCnideal}
Let $\A$ be a degenerate evolution algebra with a natural basis  $B$. If $\A\neq \rad(\A)$ and $\Gamma (\A,B)$ is connected then $V= \A \setminus \rad (\A)$ is not an ideal of $\A$.
\end{lemma}
\begin{proof} Let $B=\{e_i\}_{i \in \Lambda}$.
Firstly, by Proposition \ref{anni=spanlambdai} we have that  $\rad (\A) =\{e_i \in B \tq i \in \LLm(B)\}$, where $\m= \asi(\A)$, then $V=\span\{e_i \in B \tq i \in\Lambda\setminus\LLm(B)\}$. Since $\Gamma (\A, B)$ is connected, then there exist $j \in \Lambda \setminus \LLm(B)$ such that  $D^1(j)\cap \LLm(B) \neq \emptyset$. If  $D^1(j)\subseteq \LLm(B)$ then  $j$ is  an acyclic vertex  and therefore $j \in \LLm(B)$,  which is a contradiction. Thus $D^1(j)\not\subseteq  \LLm(B)$ and consequently   $D^1(j)\cap \left( \Lambda \setminus \LLm(B)\right) \neq \emptyset$.
Therefore we can write $e_j^2=u_1+u_2$, where $u_1 \in \rad(\A)\setminus \{0\}$ and $u_2\in V\setminus \{0 \}$. Let us suppose  that $V$ is an ideal. Then $e_j^2\in V$, which is a contradiction,  because $\A=\rad(\A)\oplus V$ (as vector space) and $u_1=e_j^2-u_2 \in \rad(\A)\cap V$, with $u_1\neq 0$.  
\end{proof}

\begin{theorem}
\label{prop:dimann=1irre}
Let $\A$ be a nilpotent evolution algebra with $\dim(\ann(\A))=1$. Then $\A$ is indecomposable.
\end{theorem} 

\begin{proof}
Let $B=\{e_i\}_{i \in \Lambda}$ be a natural basis of  $\A$ and $\LL_1(B)=\{\ell \}$. As $\A$ is nilpotent then, by Corollary  \ref{cor:a=radnilpotente} we have that  $\Lambda=\lambda_{\m}(B)$, where $\m=\asi(\A)$. 
If $j \in \Lambda \setminus \LL_1(B)$, considering Corollary \ref{cor:Lmciclo} \ref{cor:Lmciclo3}, there exist a path from $j$ to $\ell$.
Suppose that $\Gamma (\A,B)$ is  disconnected. Then there exist a  partition of $\Lambda=\Lambda'\cup \Lambda''$ such that  $D(k)\subseteq \Lambda'$ for all $k \in \Lambda'$ and  $D(k)\subseteq \Lambda''$ for all $k \in \Lambda''$. Without loss of generality we  assume that  $\ell \in \Lambda'$. Then if  $t\in \Lambda''$ there is no path from $t$ to  $\ell$, which is a contradiction. Therefore $\Gamma (\A,B)$ is connected and by \cite[Proposição 2.10]{Elduque/Labra/2015} we have that $\A$ is indecomposable.
\end{proof}

\begin{theorem} \label{teo:quocienteconexo}
Let $\A$ be a degenerate evolution algebra,  $B$  and $C$ natural basis of $\A$ and $\A/\rad(\A)$, respectively, such that $\Gamma (\A,B)$ and $\Gamma (\A/\rad(\A),C)$ are connected. If  $\rad (\A)$  is indecomposable then $\A$ is indecomposable.
\end{theorem}
\begin{proof} If $\A=\rad(\A)$ then the statement follows. In the other case, by \cite[Proposição 2.10]{Elduque/Labra/2015},  we need to prove that for any natural basis $B'=\{f_i \}_{i \in \Lambda}$ of  $\A$ the graph  $\Gamma(\A,B')$ is connected.

Let $B'=\{f_i \}_{i \in \Lambda}$ be a natural basis  of  $\A$. Let us note that since  $\A/\rad(\A)$ is a non-degenerate evolution algebra e  $\overline{B'}=\{f_i\tq  i \in\Lambda\setminus \LLm(B') \}$, where $\m=\asi(\A)$,  is a natural bases  of $\A/\rad(\A)$ then $ \Gamma (\A/\rad(\A),\overline{B})$ is also connected. Furthermore, if we consider the natural basis $B''=\{f_i \tq i \in \LLm(B') \}$  of $\rad(\A)$, as $\rad(\A)$ is indecomposable then, by \cite[Proposição 2.8]{Elduque/Labra/2015}, we have that $\Gamma(\rad(\A), B'')$ is connected. 

Let $B=\{e_i\}_{i \in \Lambda}$ a natural basis of $\A$. By Lemma \ref{lema:radCnideal} we have that $V=\rad(\A)^C=\span \{e_i \tq i \in\Lambda\setminus \LLm(B) \}=\span \{f_i \tq i \in\Lambda\setminus \LLm(B) \}$ is not an ideal of $\A$. As $V\cdot\rad(\A)=\{0 \}$, then $V$ is not  closed under multiplication, i.e. there exist $i \in \Lambda \setminus \LLm(B')$ such that 
$e_i^2 \not\in V$. Equivalently, there exist $i\in\Lambda\setminus\LLm(B')$ and $j \in \LLm(B')$ such that $j \in D^1(i)$. 
Therefore $\Gamma (\A,B')=(\Lambda,E)$ have connected subgraphs  $\Gamma (\rad(\A),B'')=(\LLm(B'),E_1)$ and $\Gamma (\A/\rad(\A),\overline{B'})=(\Lambda\setminus \LLm(B'),E_2)$  such that  $E\cap (\Lambda\setminus\LLm(B'))\times (\LLm(B'))
\neq \emptyset$ and therefore $\Gamma(\A,B')$ is connected. 
\end{proof}

\begin{example} \label{quotient_by_radical}
Let $\A$ an evolution algebra with a natural basis  $B=\{e_1,e_2,e_3,e_4,e_5\}$ such that $e_1^2=0$, $e_2^2=e_1$, $e_3^2=e_2+e_4$, $e_4^2=e_5$ e $e_5^2=e_3$. Consider the structure matrix $M_B$:
$$M_B=\left(\begin{array}{cc|ccc}
    0&0&0&0&0\\
    1&0&0&0&0\\
    \hline
    0&1&0&1&0\\
    0&0&0&0&1\\
    0&0&1&0&0
\end{array} \right).$$
\noindent Also  consider the graph $\Gamma(\A, B)$, presented in Figure \ref{fig:exaGAB}.
\begin{figure}[!h] 
\centering
\begin{tikzpicture}[scale=1]
\draw [thick] (0,0) circle (7pt);
\draw (0,0) node[font=\footnotesize] {1};
\draw [thick] (2,0) circle (7pt);
\draw (2,0) node[font=\footnotesize] {2};
\draw [thick] (4,0) circle (7pt);
\draw (4,0) node[font=\footnotesize] {3};
\draw [thick] (5.5,1) circle (7pt);
\draw (5.5,1) node[font=\footnotesize] {4};
\draw [thick] (5.5,-1) circle (7pt);
\draw (5.5,-1) node[font=\footnotesize] {5};
\draw [thick, directed]  (1.75,0) to (0.25,0)  ;
\draw [thick, directed] (3.75,0) to (2.25,0)   ;
\draw [thick, directed] (4.18,0.18) to  (5.32,0.82);
\draw [thick, directed] (5.5,0.75) to  (5.5,-0.75);
\draw [thick, directed] (5.32,-0.82) to  (4.18,-0.18);
\end{tikzpicture}
\caption{$\Gamma (\A,B)$ for the algebra $\A$ of Example \ref{quotient_by_radical}.}
\label{fig:exaGAB}
\end{figure}
As $\{1,2\}$ is the set of acyclic vertices of $\Gamma (\A,B)$, by Theorem \ref{teo:rad}, we have that  $\LLm(B)=\{1,2\}$.
Therefore $\rad(\A)=\span \{e_1,e_2\}$ and the structure matrix of $\rad(\A)$ relative to $B'=\{e_1,e_2\}$ is given by
$$M_{B'}=\left(\begin{array}{cc}
    0&0\\
    1&0
\end{array} \right),$$
and the  associated graph to $\rad(\A)$ related to $B'$ is presented in Figure \ref{fig:exarad1}.
\begin{figure}[!h] 
\centering
\begin{tikzpicture}[scale=1]
\draw (2.5,-1) node[font=\footnotesize] {1};
\draw [thick] (2.5,-1) circle (7pt);
\draw (4,-1) node[font=\footnotesize] {2};
\draw [thick] (4,-1) circle (7pt);
\draw [thick, directed]   (3.75,-1) to (2.75,-1);
\end{tikzpicture}
\caption{$\Gamma (\rad(\A),B')$ for the algebra $\A$ of Example \ref{quotient_by_radical}.}  \label{fig:exarad1}
\end{figure}
Furthermore, $\overline{B}=\{\overline{e_3}, \overline{e_4}, \overline{e_5} \}$ is a natural basis of $\A/\rad (\A)$ and $\overline{e_3}^2=\overline{e_4}$, $\overline{e_4}^2=\overline{e_5}$, $\overline{e_5}^2=\overline{e_3}$. So, the structure matrix of $\A/\rad(\A)$ related to $\overline{B}$ is given by
$$M_{\overline{B}}=\left(\begin{array}{ccc}
     0&1&0  \\
     0&0&1  \\
     1&0&0
\end{array} \right).$$
Thus, the associated graph to $\A/\rad (\A)$ related to $\overline{B}$ is the graph of Figure \ref{fig:exaA/rad}. 
\begin{figure}[!h] \label{fig:exaA/rad}
\centering
\begin{tikzpicture}[scale=1]
\draw [thick] (4,0) circle (7pt);
\draw (4,0) node[font=\footnotesize] {3};
\draw [thick] (5.5,1) circle (7pt);
\draw (5.5,1) node[font=\footnotesize] {4};
\draw [thick] (5.5,-1) circle (7pt);
\draw (5.5,-1) node[font=\footnotesize] {5};
\draw [thick, directed] (4.18,0.18) to  (5.32,0.82);
\draw [thick, directed] (5.5,0.75) to  (5.5,-0.75);
\draw [thick, directed] (5.32,-0.82) to  (4.18,-0.18);
\end{tikzpicture}
\caption{$\Gamma (\A/\rad(\A),\overline{B})$ for the algebra $\A$ of Example \ref{quotient_by_radical}.}  \label{fig:exaA/rad}
\end{figure} 
\end{example}

\begin{corollary}
Let $\A$  be a degenerate evolution algebra, $B$ and $C$ natural basis of $\A$ and $\A/\rad(\A)$, respectively, such that $\Gamma (\A,B)$ and $\Gamma (\A/\rad(\A),C)$ are connected. If $\dim(\ann(\A))=1$ then $\A$ is indecomposable.
\end{corollary}

\section{Acknowledgments}
The authors would like to thank all the participants of the  Brasil-Colombia webinar “Álgebras de evolución” for useful discussions during the preparation of this paper. T.R. thanks the UTFPR for all support provided and the use of the CCCT-CP computer facilities. P.C. thanks Pablo Rodriguez for his suggestions and comments.

\end{document}